\documentclass[10pt]{article}
\usepackage[
  margin=3cm,
  includefoot,
  footskip=30pt,
]{geometry}

\usepackage[T1]{fontenc}
\usepackage{textcomp}

\usepackage{libertinus}

\usepackage{url}
\usepackage{fullpage}
\usepackage{diagbox}

\usepackage{natbib}

\usepackage{amsthm,amsfonts,amsmath,amssymb,epsfig,color,float,graphicx,verbatim,bm,bbm}
\usepackage{enumerate}
\usepackage{enumitem}
\usepackage{wrapfig}
\usepackage{subcaption}
\usepackage[export]{adjustbox}
\usepackage{nicefrac}
\usepackage{hhline}
\usepackage{multicol}
\usepackage{multirow}
\usepackage[dvipsnames]{xcolor}

\usepackage[ruled, vlined, linesnumbered]{algorithm2e}

\usepackage{hyperref}
\hypersetup{
	colorlinks   = true, 
	urlcolor     = blue!75!black, 
	linkcolor    = blue!75!black, 
	citecolor   = blue!75!black 
}

\newtheorem{theorem}{Theorem}[section]

\newtheorem*{question*}{Question}

\newtheorem*{bigquestion*}{Big Question}

\newtheorem{lemma}[theorem]{Lemma}
\newtheorem{corollary}[theorem]{Corollary}

\newtheorem{definition}[theorem]{Definition}
\newtheorem*{definition*}{Definition}
\newtheorem{remark}[theorem]{Remark}

\newtheorem*{keyquestion*}{Key Question}

\renewcommand{\eqref}[1]{Eq.~(\ref{#1})}

\newcommand{\reals}{\mathbb{R}}

\newcommand{\E}{\mathbb{E}}

\newcommand{\half}{\frac{1}{2}}

\newcommand{\bzero}{\mathbf{0}}
\newcommand{\dist}{\mathrm{dist}}
\renewcommand{\SS}{\mathbb{S}}
\newcommand{\NN}{\mathbb{N}}

\newcommand{\ba}{\mathbf{a}}
\newcommand{\be}{\mathbf{e}}

\newcommand{\bx}{\mathbf{x}}
\newcommand{\x}{\mathbf{x}}
\newcommand{\bw}{\mathbf{w}}

\newcommand{\bg}{\mathbf{g}}
\newcommand{\g}{\mathbf{g}}

\newcommand{\bu}{\mathbf{u}}

\newcommand{\bv}{\mathbf{v}}

\newcommand{\bz}{\mathbf{z}}

\newcommand{\by}{\mathbf{y}}
\newcommand{\y}{\mathbf{y}}

\newcommand{\Ocal}{\mathcal{O}}

\newcommand{\Fcal}{\mathcal{F}}

\newcommand{\Pcal}{\mathcal{P}}

\newcommand{\Xcal}{\mathcal{X}}

\newcommand{\norm}[1]{\|#1\|}

\newcommand{\inner}[1]{\langle#1\rangle}
\newcommand{\binner}[1]{\left\langle#1\right\rangle}

\newcommand{\B}{\mathbb{B}}
\newcommand{\BB}{\mathbb{B}}

\newcommand{\conv}{\mathrm{conv}}
\newcommand{\poly}{\mathrm{poly}}

\newcommand{\eps}{\epsilon}

\newcommand{\Vcal}{\mathcal{V}}

\newcommand{\Bcal}{\mathcal{B}}
\newcommand{\Kcal}{\mathcal{K}}

\title{The Exponential Price of Determinism\\in Nonsmooth Nonconvex Optimization}
\author{Guy Kornowski}
\date{}

\begin{document}

\maketitle

\begin{abstract}
We study the complexity of finding $(\delta,\epsilon)$-Goldstein stationary points of nonsmooth nonconvex Lipschitz functions.
By now, it is known that randomized first-order algorithms can solve this task with a dimension-free oracle complexity \citep{zhang2020complexity}, whereas deterministic algorithms cannot, as their complexity must scale at least linearly with the dimension $d$ \citep{jordan2023deterministic,tian2024no}.
This leaves open whether deterministic algorithms can nevertheless solve the problem with oracle complexity polynomial in $d$.
We answer this question negatively by proving a lower bound of order $(1/\eps)^{\Omega(d)}$ for deterministic algorithm, closing the exponential gap between the previously known lower and upper bounds and resolving an open problem posed by \citet{jordan2023deterministic}.
We further discuss several extensions and implications of this result to weaker stationarity notions, finding a descent direction and deterministic smoothing.
Overall, our results establish an exponential computational advantage in nonsmooth nonconvex optimization offered by randomization.
\end{abstract}

\section{Introduction}

We consider optimization problems associated with an objective $f:\reals^d\to\reals$ which is Lipschitz continuous, but not necessarily smooth or convex. Such problems have received significant interest in recent years due to their ubiquity in machine learning related applications, where losses given by deep neural networks are generally nonsmooth and nonconvex.
Due to classic impossibility results that deem such problems generally intractable \citep{Nemirovski-1983-Problem,murty1987some}, local stationarity-based relaxations are commonly used to analyze convergence \citep{ghadimi2013stochastic,carmon2020lower}. Even so, for nonsmooth problems, such local relaxations must be defined quite carefully. Indeed, while subgradient methods converge to stationary points under mild conditions \citep{davis2020stochastic}, reaching iterates with small subgradient norm in their proximity
may require a number of queries which is exponential in the dimension $d$ \citep{kornowski2022oracle}.

In the groundbreaking work by \citet{zhang2020complexity}, the authors consider an approximate stationarity condition based on the Goldstein subdifferential \citep{goldstein1977optimization}, coined $(\delta,\eps)$-Goldstein stationarity, which requires that there exists a convex combination of gradients within a $\delta$-ball whose norm is at most $\eps$. \citet{zhang2020complexity} provided randomized first-order algorithms that converge to such points at a dimension-free rate, a result which was refined by \citet{Tian-2022-Finite,davis2022gradient}, and improved for stochastic problems by \citet{cutkosky2023optimal}.
This subsequently led to a surge of follow-up work which extended the Goldstein-stationarity framework
to
other nonsmooth nonconvex optimization settings such as zero-order \citep{lin2022gradient,chen2023faster,kornowski2024algorithm}, constrained \citep{grimmer2025goldstein,liu2024zeroth},
compositional \citep{liu2024gradient}, private \citep{zhang2023private,kornowski2025improved}, decentralized \citep{lin2024decentralized}, minimax \citep{shi2026nonsmooth},
and higher-order variants \citep{guan2026computing}.

Interestingly, algorithms considered under this framework are randomized, in sharp contrast to standard algorithms in smooth or convex optimization.
\citet{jordan2023deterministic,tian2024no} proved that this is due to an inherent obstacle: while randomized algorithms can converge to Goldstein-stationary points at a rate which is independent of the dimension $d$, any deterministic algorithm must have complexity at least linear in $d$. This established that determinism has a ``price'' in nonsmooth nonconvex optimization, which contrasts with smooth or convex optimization regimes, in which optimal algorithms are known to be deterministic.

However, the results of \citet{jordan2023deterministic,tian2024no} both suffer from the same shortcoming: 
they proved a \emph{linear}-in-$d$ lower bound for deterministic algorithms, while there is in fact no known deterministic algorithm for this task better than an exhaustive grid search,
leaving open an \emph{exponential} gap between known upper and lower bounds.
This problem was raised by \citet{jordan2023deterministic}, who asked: \textit{``Is there a deterministic first-order algorithm for nonsmooth nonconvex optimization that returns a $(\delta,\eps)$-Goldstein stationary point using $\poly(d,\delta^{-1},\eps^{-1})$ oracle calls?''}, a question that has remained open.

There are several reasons why this exponential gap is unsatisfying.
Beyond theoretical interest,
many applications of the Goldstein-stationarity framework such as various zero-order optimization settings, including most of the previously mentioned works, are known to require $\Omega(d)$ queries with or without randomization, deeming the $\Omega(d)$ lower bound for deterministic algorithms uninformative.
In such applications,
it remains unclear whether randomization is helpful at all, yet no algorithm in the current literature can do without it.

In this work, we resolve this problem, and prove that all deterministic local algorithms, a general class of iterative algorithms which contains deterministic zero- and first-order algorithms, require $T\gtrsim (1/\eps)^{\Omega(d)}$ oracle calls in order to find a Goldstein-stationary point.
This exponentially improves over the previous lower bounds for this task,
and rules out the existence of any polynomial algorithm for finding Goldstein-stationary points deterministically, thus providing a firm negative answer to the question raised by \citet{jordan2023deterministic}.
We thus establish the exponential overhead of deterministic algorithms in finding Goldstein-stationary points of nonsmooth nonconvex objectives, or equivalently, the exponential advantage offered by randomization.

The paper proceeds by introducing relevant definitions in Section~\ref{sec: prelim}, stating the main result in Section~\ref{sec: main thm} along with a proof sketch, and presenting the full proof in Section~\ref{sec: proof}.
In Section~\ref{sec: discuss} we discuss some extensions of our result and further implications to related problems.

\section{Preliminaries} \label{sec: prelim}

\paragraph{Notation.}
We let boldfaced letters (e.g., $\bx$) denote vectors, $\bzero$ is the zero vector in $\reals^d$
(where $d$ is clear from context), and $\be_1$ is the first standard basis vector. $\|\cdot\|$ will denote the Euclidean norm. $\SS^{d-1}\subset\reals^d$ denotes the unit sphere, $\BB^d(\bx,r):=\{\bz\in\reals^d:\|\bz-\bx\|\leq r\}$ is a closed ball, and we let $\BB^d:=\BB^d(\bzero,1)$.
We denote $[T]:=\{1,2,\dots,T\}$, $\dist(\bzero,A):=\inf_{\ba\in A}\|\ba\|$.

\paragraph{Nonsmooth Optimization.}

A function $f:\reals^d\to\reals$ is called $L$-Lipschitz if for any $\x,\y\in\reals^d:|f(\x)-f(\y)|\leq L\norm{\x-\y}$. We denote the function class of interest by
\[
\Fcal^d(L,\Delta):=\{f:\reals^d\to\reals~:~f\text{ is }L\text{-Lipschitz},~f(\bzero)-\inf f\leq\Delta\}~.
\]
By Rademacher's theorem, Lipschitz functions are differentiable almost everywhere (with respect to Lebesgue measure). Hence, for any function $f\in\Fcal^d(L,\Delta)$ and point $\x\in\reals^d$, the Clarke subdifferential \citep{Clarke-1990-Optimization} can be defined as
\[
\partial f(\x):=\conv\{\g\,:\,\g={\lim}_{n\to\infty}\nabla f(\x_n),\,\x_n\to \x\}~,
\]
namely, the convex hull of all limit points of $\nabla f(\x_n)$ over all sequences of differentiable points which converge to $\x$.
Given $\delta\geq 0$ the Goldstein $\delta$-subdifferential \citep{goldstein1977optimization} of $f$ at $\x$ is the set
\[
\partial_{\delta}f(\x):=\conv\left({\bigcup}_{\y\in \B^d(\bx,\delta)}\partial f(\y)\right)~,
\]
namely, all convex combinations of Clarke subgradients at points in the $\delta$-ball around $\bx$.

\begin{definition}
A point $\bx$ is called a $(\delta,\eps)$-Goldstein stationary point of $f$ if $\dist(\bzero,\partial_{\delta} f(\x))\leq\eps$.
\end{definition}

\paragraph{Local Algorithms.}
We consider iterative algorithms, in the standard oracle complexity framework \citep{Nemirovski-1983-Problem}:
At each time step $t\in[T]$, an algorithm produces $\bx_t$ based on previously observed oracle responses $(\Ocal_f(\bx_1),\dots,\Ocal_f(\bx_{t-1}))$,
and receives the oracle response $\Ocal_f(\bx_t)$.
An oracle is called \emph{local} if for any $\bx$ and any two functions $f,g$ that are equal over some neighborhood of $\bx$, it holds that $\Ocal_{f}(\bx)=\Ocal_{g}(\bx)$.
A canonical example is the first-order oracle $\Ocal_{f}(\bx)=(f(\bx),\partial f(\bx))$,\footnote{For the sake of proving a lower bound, assuming an algorithm can access the entire subdifferential, even though this is typically not the case, only makes the result stronger.} but one can consider other oracles such as those which return derivatives of higher orders, whenever they exist.

\section{Main Theorem} \label{sec: main thm}

We now present the main result of this paper.

\begin{theorem} \label{thm: main}
Let $L,\Delta>0$.
For any $\delta<{\Delta}/{L},~\eps<{L}/{32}$
and $d\geq \lceil 2\log_2(L/8\eps)\rceil$,
no deterministic local algorithm applied to $f\in\Fcal^d(L,\Delta)$ can guarantee that any of its $T$ iterates is a $(\delta,\eps)$-Goldstein stationary point of $f$,
unless
\[
T>\left(\frac{L}{32\eps}\right)^{d/2}~.
\]
\end{theorem}

An immediate corollary of Theorem~\ref{thm: main} is that there is no deterministic first-order algorithm for nonsmooth nonconvex Lipschitz  optimization that returns a $(\delta,\eps)$-Goldstein stationary point using $\poly(d,\delta^{-1},\eps^{-1})$ oracle
calls, which answers an open question raised by \citet{jordan2023deterministic}.

We further note that the function $f\in\Fcal^d(L,\Delta)$ constructed in the proof of Theorem~\ref{thm: main} is not pathological (cf. \citealp{daniilidis2020pathological}), and it can be assumed to satisfy additional structure beyond Lipschitzness. In particular, in Section~\ref{sec: regular} we discuss how the result remains true even for continuously differentiable $f$, which implies various common regularity assumptions in nonsmooth optimization.

Before moving on to the formal proof in the next section, we will sketch the proof idea, and how it differs from the previous $\Omega(d)$ lower bounds.
We focus on $L=\Delta=1$ for simplicity, and we start the discussion by recalling the main idea behind the prior proofs of \cite{jordan2023deterministic,tian2024no}: If at each time step $t\in[T]$ a deterministic algorithm queries $\bx_{t}$, and learns from the local oracle that the function $f$ locally equals $f_t(\bx)= \inner{\be_1,\bx-\bx_t}$, this fixes the algorithm's iterates $\Xcal:=\{\bx_1,\dots,\bx_T\}$. Hence, it suffices to construct for any set $\Xcal\subset\reals^d$ of size $T$, a function which around each point $\bx_t\in\Xcal$ locally matches $f_t$, without any point in $\Xcal$ being Goldstein-stationary. To do so, these prior works fix some unit direction $\bw\perp \Xcal$ orthogonal to all iterates, and construct a function which globally equals $f_\bw(\bx):=\max\{\inner{\bw,\bx},-1\}$, with local interpolations around each point in $\Xcal$ to match all the $f_t$'s. It is easy to see that all Goldstein-stationary points of $f_\bw$ are anti-correlated with $\bw$, ruling out all points in $\Xcal$.
A direction $\bw\perp \Xcal$ clearly exists whenever $|\Xcal|=T<d$, leading to the $\Omega(d)$ lower bound.
It is also clear that this strategy breaks as soon as $T=d$ since no direction $\bw\in\SS^{d-1}$ orthogonal to all iterates is guaranteed to exist, hence failing to prove any lower bound larger than $d$.

The key observation behind our exponential improvement is that exact orthogonality is unnecessarily strong, and that it is possible to construct a function $f$ with the desired properties based on a direction $\bw\in\SS^{d-1}$ which is merely not too aligned with $\Xcal$. Specifically, after translating so that $\bx_1=\bzero$, consider the set $\Vcal:=\{(\bx_i-\bx_j)/\|\bx_i-\bx_j\|:i\neq j\in[T]\}$ consisting of all normalized points in $\Xcal$ as well as their normalized differences.
For any $\bv\in\Vcal$, the subset of points in $\SS^{d-1}$ which are $\eps$-close to $\bv$ is a spherical cap, whose spherical measure decays with $d$ as $\eps^{d-1}$. Since $|\Vcal|=O(T^2)$, a union bound shows that as long as $T\lesssim (1/\eps)^{\Omega(d)}$, there is some $\bw\in\SS^{d-1}$ which is not too correlated, or anti-correlated, with any $\bv\in\Vcal$, even after conditioning on $\inner{\be_1,\bw}\geq2\eps$; see Figure~\ref{fig: f} (left) for an illustration.
With such $\bw$ in hand, we replace the previous half-space construction $\bx\mapsto \max\{\inner{\bw,\bx},-1\}$ which resulted in all gradients outside the minimizer set being exactly $\bw$, by a more flexible construction, restricting gradients to the cap $\Kcal_\bw=\{\bg\in\BB^d:\inner{\bg,\bw}\geq2\eps\}$. Since $\be_1\in\Kcal_\bw$, this restriction is compatible with the local oracle responses near every $\bx_t\in\Xcal$, and moreover, $\Kcal_\bw$ is convex and bounded away from the origin, so any point whose nearby gradients remain in $\Kcal_\bw$ cannot be Goldstein-stationary. The fact that the differences in $\Vcal$ are not correlated with $\bw$ ensures that 
we can ``pinch'' a small affine piece with slope $\be_1$ around each iterate and extend it downward along the $-\bw$ direction without any two such pieces colliding;
see Figure~\ref{fig: f} (right) for an illustration. The formal construction in our proof
implements this gluing idea through an extension of the convex-analytic support function associated to $\Kcal_\bw$, which preserves the prescribed local behavior around all points in $\Xcal$ while keeping every Goldstein subgradient at the query points away from $\bzero$.

\begin{figure}[t]
    \centering

    \begin{subfigure}{0.48\textwidth}
        \centering
       \includegraphics[width=0.9\linewidth, trim=1cm 0 12cm 0, clip]{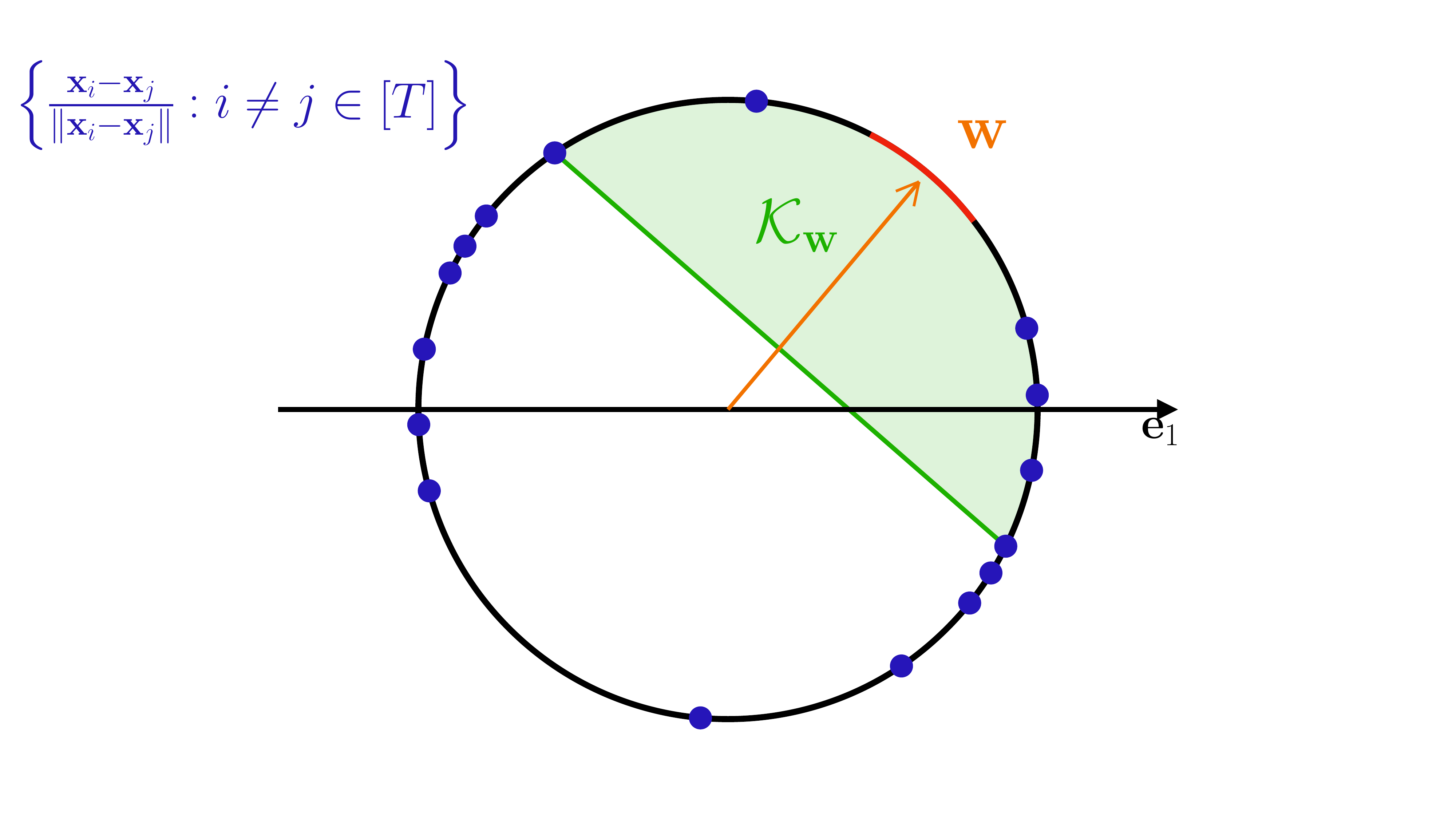}
        \caption{The key quantities playing a role in the construction: the set of normalized differences, a sufficiently separated direction $\bw\in\SS^{d-1}$, and its associated cap $\Kcal_\bw\subset\BB^{d}$ of
        allowed gradients.}
        \label{fig:first}
    \end{subfigure}
    \hfill
    \begin{subfigure}{0.48\textwidth}
        \centering
        \includegraphics[width=0.75\linewidth]{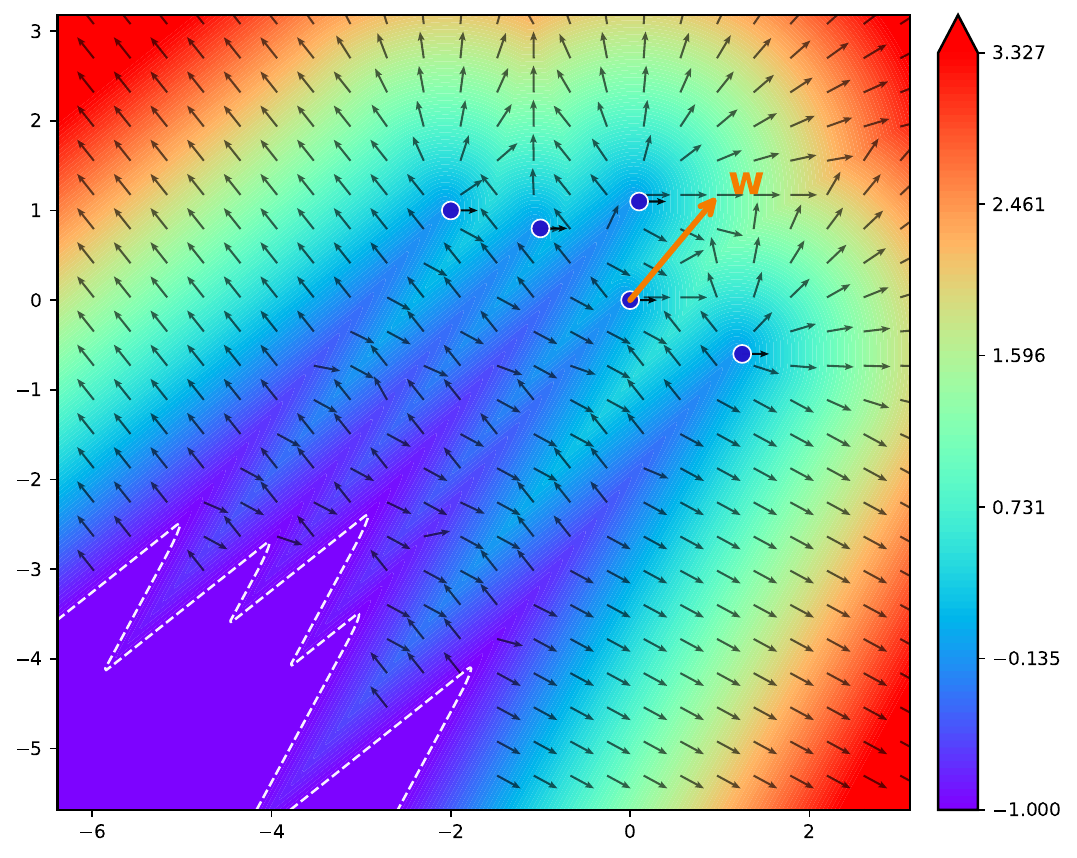}
        \caption{Contour plot of $f:\reals^2\to\reals$ associated to iterates (blue dots); arrows represent the gradient field, and the dashed white line is the boundary of $\arg\min f$, connected to each iterate in the direction of $\bw$.}
        \label{fig:second}
    \end{subfigure}

    \caption{Illustration of the function $f$ constructed in the proof of Theorem~\ref{thm: main}.}
    \label{fig: f}
\end{figure}

\section{Proof of Theorem~\ref{thm: main}} \label{sec: proof}

We follow a resisting oracle strategy: Suppose that at any time $t\in[T-1]$, a deterministic local algorithm picks $\bx_t\in\reals^d$, and the local oracle's response is $\Ocal_{f_t}(\bx_t)$ where $f_t(\bx):=\inner{L\be_1,\bx-\bx_t}$. Namely, the oracle responds that the function is locally equal to $\bx\mapsto\inner{L\be_1,\bx-\bx_t}$. 
By induction on $t$, since the algorithm is deterministic, this fixes all iterates $\bx_1,\dots,\bx_T$. Our goal is to construct a single function $f\in\Fcal^d(L,\Delta)$ matching all of these oracle responses, such that no $\bx_t$ is a $(\delta,\eps)$-Goldstein stationary point. We assume without loss of generality that $\bx_1=\bzero$, that all $T$ iterates are distinct, and denote $\Xcal:=\{\bx_1,\dots,\bx_T\}\subset\reals^d$. We also denote $\eps'=2\eps/L$.

The following lemma ensures that as long as $T$ isn't exponentially large with respect to $d$, roughly up to $T^2\lesssim \eps^{-d}$, it is always possible to find a direction $\bw\in\SS^{d-1}$ whose correlation to normalized differences between points in $\Xcal$ is bounded away from $1$.

\begin{lemma}\label{lem: sphere cap}
If $T\leq(16\eps')^{-d/2}$, then there exists a unit vector $\bw\in\SS^{d-1}$ such that $\inner{\bw,\be_1}\geq \eps'$ and for all $i\neq j\in[T]:$
\begin{equation}\label{eq: w cond}
\left|\binner{\bw,\frac{\bx_i-\bx_j}{\|\bx_i-\bx_j\|}}\right|<\sqrt{1-\eps'^2}~.
\end{equation}

\end{lemma}

\begin{proof}[Proof of Lemma~\ref{lem: sphere cap}]

With slight abuse of notation, we let $\SS^{d-2}\subset \be_1^{\perp}$ denote the unit sphere of the subspace orthogonal to $\be_1$.
For any $ \bv\in\SS^{d-1}\setminus\{\pm\be_1\}$, let
\[
A_\bv:=\left\{\bz\in\SS^{d-2}\subset\be_1^\perp~:~\left|\binner{\eps'\be_1+\sqrt{1-\eps'^2}\bz,\bv}\right|\geq\sqrt{1-\eps'^2}\right\}~.
\]
We start by bounding the spherical surface measure of $A_\bv$. To that end, decompose
\[
\bv=\gamma_v\be_1+\sqrt{1-\gamma_v^2}\bu_v
\]
where $\gamma_v\in(-1,1),~\bu_v\in\SS^{d-2}\subset\be_1^\perp$, and let $\alpha:=\frac{\eps'}{\sqrt{1-\eps'^2}},~\beta:=\sqrt{1-\gamma_v^2}$. Note that
\[
\binner{\eps'\be_1+\sqrt{1-\eps'^2}\bz,\bv}=\eps'\gamma_v+\sqrt{1-\eps'^2}\beta\inner{\bz,\bu_v}~,
\]
and therefore
\[
\binner{\eps'\be_1+\sqrt{1-\eps'^2}\bz,\bv}\geq\sqrt{1-\eps'^2}
~~\implies~~
\inner{\bz,\bu_v}\geq
\frac{\sqrt{1-\eps'^2}-\eps'\gamma_v}{\beta\sqrt{1-\eps'^2}}
=\frac{1-\alpha\gamma_v}{\sqrt{1-\gamma_v^2}}
\geq \sqrt{1-\alpha^2}~,
\]
where the last inequality holds since
$(1-\alpha\gamma_v)^2-(1-\alpha^2)(1-\gamma_v^2)=(\gamma_v-\alpha)^2\geq 0$. 
A similar calculation shows that $\inner{\eps'\be_1+\sqrt{1-\eps'^2}\bz,\bv}\leq-\sqrt{1-\eps'^2}$ implies $-\inner{\bz,\bu_v}\geq \sqrt{1-\alpha^2}$. 
Hence, we get that
\begin{equation} \label{eq: caps}
A_\bv\subset 
\left\{\bz\in\SS^{d-2}\subset\be_1^\perp~:~\binner{\bz,\bu_v}\geq\sqrt{1-\alpha^2}\right\}
\cup
\left\{\bz\in\SS^{d-2}\subset\be_1^\perp~:~\binner{\bz,-\bu_v}\geq\sqrt{1-\alpha^2}\right\}
~.
\end{equation}

Next, we claim that each of the two spherical caps in \eqref{eq: caps} has Euclidean radius at most $2\eps'$. Indeed, if $\inner{\bz,\bu_v}\geq\sqrt{1-\alpha^2}$ then
\[
\|\bz-\bu_v\|^2=2(1-\inner{\bz,\bu_v})\leq 2(1-\sqrt{1-\alpha^2})\leq 2\alpha^2=\frac{2\eps'^2}{1-\eps'^2}<4\eps'^2~,
\]
and the second cap follows by symmetry.
Thus, by a standard bound on the surface measure of a spherical cap (Lemma~\ref{lem: cap measure}), we get that
\[
\mu_{d-2}(A_\bv)\leq 2(4\eps')^{d-2}~,
\]
where $\mu_{d-2}$ is the normalized $(d-2)$-dimensional surface measure on $\SS^{d-2}\subset \mathbf{e}_1^\perp$.

Union bounding over $\bv\in\left\{(\bx_i-\bx_j)/{\|\bx_i-\bx_j\|}:~1\leq i<j\leq T\right\}$ which is a set of size at most $\frac{T(T-1)}{2}<\frac{T^2}{2}$, we see that the set of $\bw\in\SS^{d-1}$ which do \emph{not} satisfy the conclusion of the lemma have normalized $(d-1)$-dimensional surface measure of at most
\[
T^2(4\eps')^{d-2} \leq
\left(\frac{1}{16\eps'}\right)^{d}(4\eps')^{d-2}
=\frac{1}{16\eps'^2 4^{d}}
\overset{[\,d>\log_2(1/4\eps')\,]}{<}1~.
\]
Thus, there must exist some $\bw\in\SS^{d-1}$ in the complementary set since it has positive measure, satisfying the required properties.
\end{proof}

From now on, fix $\bw\in\SS^{d-1}$ to be the unit vector given by Lemma~\ref{lem: sphere cap}, and define the convex set
\[
\Kcal_\bw:=\{\bg\in\BB^d~:~\inner{\bg,\bw}\geq \eps'\}\subset\BB^d~.
\]
Note that ${\Kcal_{\bw}}$ satisfies
\begin{equation} \label{eq: Kw dist}
\dist(\bzero,{\Kcal_{\bw}})
=\inf_{\bg\in {\Kcal_{\bw}}}\|\bg\|
\geq \inf_{\bg\in {\Kcal_{\bw}}}\inner{\bg,\bw}
\geq \eps'~.
\end{equation}
Let $\sigma_{\Kcal_{\bw}}:\reals^d\to\reals$ be its associated support function
\[
\sigma_{\Kcal_{\bw}}(\bv):=\sup_{\bg\in {\Kcal_{\bw}}}\inner{\bg,\bv}~.
\]

\begin{lemma} \label{lem: sigma>0}
    For all $i\neq j\in[T]:~\sigma_{\Kcal_{\bw}}(\bx_i-\bx_j)>0$.
\end{lemma}

\begin{proof}[Proof of Lemma~\ref{lem: sigma>0}]
Fix $i\neq j\in[T]$, and denote $\bv:=\bx_i-\bx_j\neq \bzero$. Decompose $\bv= \inner{\bv,\bw}\bw+\bu$ where $\bu\perp \bw$.
By \eqref{eq: w cond} it holds that $|\inner{\bv,\bw}|^2<(1-\eps'^2)\|\bv\|^2
=(1-\eps'^2)(|\inner{\bv,\bw}|^2+\|\bu\|^2)$, thus
\begin{equation} \label{eq: u>0}
(1-\eps'^2)\|\bu\|^2>\eps'^2|\inner{\bv,\bw}|^2
\geq 0~,
\end{equation}
and in particular $\bu\neq \bzero$.
Let $\bg:=\eps'\bw+\frac{\sqrt{1-\eps'^2}}{\|\bu\|}\bu$ which is a unit vector that satisfies $\inner{\bg,\bw}=\eps'$, hence $\bg\in {\Kcal_{\bw}}$. Therefore, it holds that
\begin{align*}
\sigma_{\Kcal_{\bw}}(\bv)\geq
\inner{\bg,\bv}
=\binner{\eps'\bw+\tfrac{\sqrt{1-\eps'^2}}{\|\bu\|}\bu,\inner{\bv,\bw}\bw+\bu}
=\eps'\inner{\bv,\bw}+\sqrt{1-\eps'^2}\|\bu\|
>0~,
\end{align*}
the last inequality holding due to \eqref{eq: u>0}.
\end{proof}

Following the lemma, denote
\begin{align*}
D:={\min}_{i\neq j\in[T]}~\sigma_{\Kcal_{\bw}}(\bx_i-\bx_j)>0~,
~~~~~\rho:={\min}_{i\neq j\in[T]}~\|\bx_i-\bx_j\|>0~,
~~~~~r:=\frac{1}{8}\min\{\rho,\,D,\,1,\,\Delta/L\}>0~,
\end{align*}
and let
\[
\Bcal:=\bigcup_{t\in[T]}\BB^d(\bx_t,r)~.
\]
Denote by $\Pi_\Xcal:\Bcal\to\Xcal$ the projection onto $\Xcal$. Note that the projection is well-defined since $\Bcal$ is a disjoint union of balls around points in $\Xcal$ of the same radius, so it simply maps each point to its corresponding ball's center. We define the function
\[
F(\bx):=\inf_{\bz\in \Bcal}\left\{
\binner{\be_1,\bz-\Pi_{\Xcal}(\bz)}
+\sigma_{\Kcal_{\bw}}(\bx-\bz)\right\}
~.
\]

$F$ defined above is essentially our function construction, which we will later scale by $L$ and truncate at $-\Delta$ to ensure the required Lipschitz parameter and initial suboptimality gap. The next lemma establishes its useful properties, namely that it is Lipschitz, has a subdifferential restricted to $\Kcal_\bw$, and is consistent with the oracle responses given by the resisting strategy (up to scaling).

\begin{lemma} \label{lem: F}
The following hold:
\begin{enumerate}[label=(\roman*)]
    \item $F$ is $1$-Lipschitz.
    \item $\partial F(\bx)\subseteq {\Kcal_{\bw}}$ for all $\bx\in\reals^d$. Hence, for any $\delta>0:~\partial_\delta F(\bx)\subseteq {\Kcal_{\bw}}$.
\item For any~ $\bx\in\Bcal:~F(\bx)=\inner{\be_1,\bx-\Pi_\Xcal(\bx)}$. Thus, for $\bx\in\BB^d(\bx_t,r):~F(\bx)=\inner{\be_1,\bx-\bx_t}$.
\end{enumerate}

\end{lemma}

\begin{proof}[Proof of Lemma~\ref{lem: F}]

For the first item, fix $\bx,\by\in\reals^d$. 
Since $\Kcal_\bw\subset\BB^d$, its support function $\sigma_{\Kcal_\bw}$ is $1$-Lipschitz by Cauchy-Schwarz, thus for every $\bz\in\Bcal$ it holds that
\[
\binner{\be_1,\bz-\Pi_{\Xcal}(\bz)}
+\sigma_{\Kcal_{\bw}}(\bx-\bz)
\leq \binner{\be_1,\bz-\Pi_{\Xcal}(\bz)}+\sigma_{\Kcal_{\bw}}(\by-\bz)+\|\bx-\by\|~.
\]
Taking the infimum over $\bz\in\Bcal$ yields
\[
F(\bx)\leq F(\by)+\|\bx-\by\|~.
\]
By symmetry, $F(\by)\leq F(\bx)+\|\bx-\by\|$ and so $|F(\by)- F(\bx)|\leq \|\bx-\by\|$, namely $F$ is $1$-Lipschitz. Note that in particular, $\partial F$ is well-defined.

We turn to proving the second item. Note that the support function $\sigma_{\Kcal_{\bw}}$ is sub-additive since the sup operation is. Hence, for any $\bx,\bv\in\reals^d$ and $h>0$ it holds that
\begin{align*}
F(\bx+h\bv)&=\inf_{\bz\in \Bcal}\left\{
\binner{\be_1,\bz-\Pi_{\Xcal}(\bz)}
+\sigma_{\Kcal_{\bw}}(\bx+h\bv-\bz)\right\}
\\&\leq\inf_{\bz\in \Bcal}\left\{
\binner{\be_1,\bz-\Pi_{\Xcal}(\bz)}
+\sigma_{\Kcal_{\bw}}(\bx-\bz)+\sigma_{\Kcal_{\bw}}(h\bv)\right\}
\\&=F(\bx)+\sigma_{\Kcal_{\bw}}(h\bv)
\\&=F(\bx)+h\cdot\sigma_{\Kcal_{\bw}}(\bv)~.
\end{align*}
Hence, for any $\bx$ at which $F$ is differentiable, and any $\bv\in\reals^d:$
\[
\inner{\nabla F(\bx),\bv}
=\lim_{h\to 0^+}\frac{F(\bx+h\bv)-F(\bx)}{h}
\leq\sigma_{\Kcal_{\bw}}(\bv)~.
\]
If it were the case that $\nabla F(\bx)\notin {\Kcal_{\bw}}$, then since ${\Kcal_{\bw}}$ is compact and convex, there would have existed a separating hyperplane $\bv'\in\reals^d$ such that $\inner{\nabla F(\bx),\bv'}>\sup_{\bg\in {\Kcal_{\bw}}}\inner{\bg,\bv'}=\sigma_{\Kcal_{\bw}}(\bv')$, in contradiction. 
Thus, at all differentiable points $\nabla F(\bx)\in {\Kcal_{\bw}}$, and since ${\Kcal_{\bw}}$ is closed and convex, we conclude by definition of the Clarke subdifferential that $\partial F\subseteq {\Kcal_{\bw}}$ at all points.
Once again since $\Kcal_\bw$ is a convex set, we further get $\partial_\delta F\subseteq {\Kcal_{\bw}}$, proving the second item.

We move on to prove the third item. Let $F_\Bcal:\Bcal\to\reals$
be $F_\Bcal(\bx):=\binner{\be_1,\bx-\Pi_{\Xcal}(\bx)}$, and note that
$|F_\Bcal|\leq r$ by Cauchy-Schwarz.
We start by showing that for any $\bx,\bz\in\Bcal:$
\begin{equation} \label{eq: F<sigma}
F_{\Bcal}(\bx)-F_{\Bcal}(\bz)\leq\sigma_{\Kcal_{\bw}}(\bx-\bz)~.
\end{equation}
Indeed, if $\Pi_\Xcal(\bx)=\Pi_\Xcal(\bz)$ then $F_{\Bcal}(\bx)-F_{\Bcal}(\bz)=\inner{\be_1,\bx-\bz}\leq\sigma_{\Kcal_{\bw}}(\bx-\bz)$, where the inequality holds since $\be_1\in {\Kcal_{\bw}}$. Otherwise, if $\Pi_\Xcal(\bx)\neq\Pi_\Xcal(\bz)$, then since $\sigma_{\Kcal_{\bw}}$ is $1$-Lipschitz as ${\Kcal_{\bw}}\subset\BB^d$, it holds that
\begin{align*}
\sigma_{\Kcal_{\bw}}(\bx-\bz)
&\geq\sigma_{\Kcal_{\bw}}(\Pi_\Xcal(\bx)-\Pi_\Xcal(\bz))-\|(\bx-\bz)-(\Pi_\Xcal(\bx)-\Pi_\Xcal(\bz))\|
\\&\geq D-\|\bx-\Pi_\Xcal(\bx)\|-\|\bz-\Pi_\Xcal(\bz)\|
\\&\geq D-2r> 2r
\geq F_{\Bcal}(\bx)-F_{\Bcal}(\bz)~.
\end{align*}

Having established \eqref{eq: F<sigma}, we now prove the lemma by showing that $F(\bx)=F_\Bcal(\bx)$ for any $\bx\in\Bcal$. On one hand,
\[
F(\bx)=
\inf_{\bz\in \Bcal}\left\{
\binner{\be_1,\bz-\Pi_{\Xcal}(\bz)}
+\sigma_{\Kcal_{\bw}}(\bx-\bz)\right\}
\leq 
\binner{\be_1,\bx-\Pi_{\Xcal}(\bx)}
+\underset{=0}{\underbrace{\sigma_{\Kcal_{\bw}}(\bzero)}}
=F_\Bcal(\bx)~.
\]
On the other hand, by \eqref{eq: F<sigma}, for any $\bz\in\Bcal~:F_\Bcal(\bz)+\sigma_{\Kcal_{\bw}}(\bx-\bz)\geq F_\Bcal(\bx)$, and taking the infimum over $\bz$ gives $F(\bx)\geq F_\Bcal(\bx)$. Overall we get that $F(\bx)=F_\Bcal(\bx)$, completing the proof.

\end{proof}

Finally, given $L,\Delta>0$, let
\[
f(\bx):=\max\left\{L\cdot F(\bx),\,-\Delta\right\}~.
\]
We see that $f$ is $L$-Lipschitz since $F$ is $1$-Lipschitz, and satisfies
\[
f(\bzero)-\inf f=\max\{L\cdot 0,-\Delta\}-(-\Delta)= \Delta
~,
\]
hence $f\in\Fcal^d(L,\Delta)$.

Furthermore, for any $\bx\in\Xcal$ and $\delta<\frac{\Delta}{L}$, if $\by\in\BB^d(\bx,\delta)$ then the Lipschitz condition ensures that $f(\by)\geq f(\bx)-L\delta=-L\delta>-\Delta$, and so $f|_{\BB^d(\bx,\delta)}\equiv L\cdot F|_{\BB^d(\bx,\delta)}$, further implying that $\partial_\delta f(\bx)=L\cdot\partial_\delta F(\bx)$. In particular, following Lemma~\ref{lem: F}, we see that
\[
\dist(\bzero,\partial f_\delta(\bx))=L\cdot\dist(\bzero,\partial F_\delta(\bx))
\geq L\cdot \dist(\bzero,\Kcal_\bw)
\overset{\text{\eqref{eq: Kw dist}}}{\geq} L\eps'
>\eps~,
\]
and so
$\bx$ is not a $(\delta,\eps)$-Goldstein stationary point of $f$. We also see that the oracle responses are consistent with the resisting strategy in hindsight, since for any $\bx\in\BB^d(\bx_t,r):$
\[
f(\bx)= \max\{L\cdot F(\bx) ,-\Delta \}
= \max\{L\inner{\be_1,\bx-\bx_t} ,-\Delta \}
\overset{[r<\Delta/L]}{=} \inner{L\be_1,\bx-\bx_t}
=f_t(\bx)
\]
which completes the proof.

\section{Extensions and Implications} \label{sec: discuss}

In this concluding section, we discuss some extensions and implications of our main result.

\subsection{Regularity} \label{sec: regular}

As mentioned earlier, Theorem~\ref{thm: main} (and all of its corollaries to follow) remains true under further regularity assumptions on $f$, even as far as continuous differentiability.
In this sense, the hardness is not due to pathological subgradient behavior (cf. \citealp{daniilidis2020pathological}).
To see this, let $\eta>0$ be sufficiently small, and consider the mollified function
\[
\tilde{f}(\bx)
:=\E_{\bz\sim\mathrm{Unif}(\BB^d(\bzero,\eta))}[f(\bx+\bz)]
=\frac{1}{\mathrm{vol}(\BB^d(\bzero,\eta))}\int_{\BB^d(\bzero,\eta)}f(\bx+\bz)d\bz
~,
\]
which is $C^1$ smooth. Note that by Lemma~\ref{lem: F}, locally around each query point $\bx_t$ it holds that
\begin{align*}
\tilde{f}(\bx)
&=\frac{1}{\mathrm{vol}(\BB^d(\bzero,\eta))}\int_{\BB^d(\bzero,\eta)}\inner{L\be_1,\bx-\bx_t+\bz}d\bz
\\&=\frac{1}{\mathrm{vol}(\BB^d(\bzero,\eta))}\int_{\BB^d(\bzero,\eta)}\inner{L\be_1,\bx-\bx_t}d\bz+
\frac{1}{\mathrm{vol}(\BB^d(\bzero,\eta))}\int_{\BB^d(\bzero,\eta)}\inner{L\be_1,\bz}d\bz
\\&=\inner{L\be_1,\bx-\bx_t}+0
\\&=f(\bx)~,
\end{align*}
and so a local algorithm applied to $\tilde{f}$ results in the same iterates as if it were applied to $f$. We further argue that no $\bx_t$ is a $(\delta,\eps)$-Goldstein stationary point of $\tilde{f}$. Indeed, if it were, then \citet[Lemma 4]{kornowski2024algorithm} implies that it is also a $(\delta+\eta,\eps)$-Goldstein stationary point of $f$, thus contradicting Theorem~\ref{thm: main} if we set $\eta$ sufficiently small so that $\delta+\eta<\Delta/L$.

\subsection{Hardness of a single Goldstein-descent step} \label{sec: single step}

We would like to point out that what we have proved is in fact stronger than hardness of finding a Goldstein-stationary point (as it is phrased in Theorem~\ref{thm: main}). To this end, we
recall that Goldstein stationarity-based analyses date back to
Goldstein's descent lemma, which can be phrased as follows:
\begin{lemma}[\citealp{goldstein1977optimization}]
For any $f\in\Fcal^d(L,\Delta)$ and
$\bx\in\reals^d$, let $\bg_\bx$ be the minimum-norm element in $\partial_\delta f(\bx)$. If $\bg_\bx\neq\bzero$, then

\begin{equation}\label{eq: goldstein descent}
f\left(\bx-\delta\frac{\bg_\bx}{\|\bg_\bx\|}\right)\leq f(\bx)-\delta\|\bg_\bx\|
~.
\end{equation}

\end{lemma}

Goldstein's lemma motivates the conceptual algorithm $\bx_{t+1}=\bx-\delta\frac{\bg_\bx}{\|\bg_\bx\|}$, requiring the computation of $\bg_\bx$ at each step, or more feasibly, the approximation thereof so that \eqref{eq: goldstein descent} applies and a descent is found.
Notably, this is the starting point of the design of gradient sampling methods
\citep{Burke-2002-Approximating,Burke-2005-Robust,Burke-2020-Gradient,Kiwiel-2007-Convergence}.

The proof of Theorem~\ref{thm: main} establishes more than the lack of stationarity of the iterates $\bx_1,\dots,\bx_T$ for any deterministic algorithm in the worst case, but moreover,
that they all satisfy $f(\bx_t)=0$ by Lemma~\ref{lem: F} $(iii)$. Therefore, we get:

\begin{corollary}
No deterministic local algorithm can, given a point $\bx$, guarantee finding a single descent direction $\bg_\bx$ which satisfies \eqref{eq: goldstein descent}  unless $T\gtrsim \exp(\Omega(d))$.
\end{corollary}

\subsection{Relaxed stationarity}

Our hardness result further applies to a recently proposed relaxed stationarity notion which is even easier to satisfy.
Specifically, \citet{zhang2024randomscaling} defined a relaxation of Goldstein stationarity which allows gradients to be combined beyond a ball. Formally, let $\lambda>0$, and denote
\[
\mathcal{P}_\lambda f(\bx):=\inf_{\substack{\text{random vector}~\by\in\reals^d\\\bg\in\partial f(\by)\text{~a.s.}}}\left\{\|\E[\bg]\|+\lambda\cdot\E\|\bx-\by\|^2\right\}
~.
\]
A point is called a $(\lambda,\epsilon)$-stationary point of $f$ if $\Pcal_\lambda f(\bx)\leq\eps$.\footnote{We note that \citet{zhang2024randomscaling} additionally required in the definition of $\Pcal_\lambda$ that $\E[\by]=\bx$. Dropping this requirement only strengthens our lower bound, and allows a more direct comparison to Goldstein stationarity.}
Note that by restricting to distributions over $\by$
supported on $\BB(\bx,\delta)$, it is clear that $\Pcal_\lambda f(\bx)\leq\dist(\bzero,\partial_\delta f(\bx))+\lambda\delta^2$, and so any $(\sqrt{\eps/2\lambda},\eps/2)$-Goldstein stationary point is $(\lambda,\eps)$-stationary, showing that this stationarity notion is indeed a more relaxed one.
This notion has found use in analyzing randomized variants of various algorithms which are popular in machine learning practice, such as Adam \citep{ahn2024adam}, schedule-free methods \citep{ahn2025general},
and spectral optimizers \citep{jiang2026adaptive,li2026muon}.

The proof of Theorem~\ref{thm: main} can be used to further show the same exponential lower bound for deterministically finding $(\lambda,\eps)$-stationary points.
Indeed, suppose $\eps<\lambda\Delta^2/2L^2$.
The proof considers $f=\max\{L\cdot F,-\Delta\}\in\Fcal^d(L,\Delta)$, and by Lemma~\ref{lem: F} it holds that $\partial F(\by)\subseteq \Kcal_\bw$ for all $\by\in\reals^d$, and $f(\bx_t)=\max\{0,-\Delta\}=0$ for all $t\in[T]$.
Therefore, for any $t\in[T]$, any vector $\by\in\reals^d$ and $\bg\in\partial f(\by)$, if $\|\by-\bx_t\|<\Delta/L$ then
\[
f(\by)> f(\bx_t)-\Delta=-\Delta
~~~\implies~~~
\partial f(\by)=\partial F(\by)\subseteq\Kcal_\bw
~~~\implies~~~
\binner{\bg,\bw}\geq 2\eps~,
\]
whereas if $\|\by-\bx_t\|\geq \Delta/L$ then
\[
\lambda\|\bx_t-\by\|^2\geq 2\eps~,
\]
so overall for any random vector $\by$ and $\bg\in\partial f(\by):$
\[
\|\E[\bg]\|+\lambda\E\|\bx_t-\by\|^2
\geq
\E\left[\inner{\bg,\bw}+\lambda\E\|\bx_t-\by\|^2\right]
\geq 2\eps~.
\]
Taking the infimum over all such random vectors proves that $\Pcal_\lambda f(\bx_t)\geq2\eps$ for all $t\in[T]$, and so we get:

\begin{corollary}
No deterministic local algorithm can guarantee finding a $(\lambda,\eps)$-stationary point unless $T\gtrsim \exp(\Omega(d))$.
\end{corollary}

\subsection{Deterministic smoothing}

Our result also has implications for the complexity of deterministically
smoothing nonsmooth functions. Following
\citet[Definition~6]{jordan2023deterministic}, consider a smoothing algorithm with complexity $T$ to be an algorithm which, given oracle access to a function $f\in\Fcal^d(L,\Delta)$, computes
$(\tilde f(\bx),\nabla\tilde f(\bx))$ at any requested point $\bx$
using at most $T$ oracle calls, for some smooth function $\tilde f$.
The smoothing is called \emph{meaningful} if $\nabla\tilde f$ is
$L_1$-Lipschitz with $L_1\leq\mathrm{poly}(d)$,
$\tilde{f}(\bzero)-\inf\tilde{f}\leq\mathrm{poly}(d)$
and any
$(\delta,\eps)$-Goldstein stationary point of $\tilde f$ is a
$(p_1(\delta,\eps),p_2(\delta,\eps))$-Goldstein stationary point of $f$,
where $p_1,p_2$ are fixed, dimension-independent polynomials tending
to zero as $(\delta,\eps)\to(0,0)$. Thus, approximate stationarity
of the smoothed function provides an approximate-stationarity
guarantee for the original function, allowing a reduction to smooth optimization. It should be noted that there exist simple and efficient randomized smoothing methods satisfying these guarantees (in expectation) \citep{flaxman2005online,yousefian2012stochastic,duchi2012randomized,shamir2017optimal}, whereas \citet[Theorem 7]{jordan2023deterministic} proved that no such deterministic algorithm has dimension-free complexity.

As a corollary of our result, we can strengthen this conclusion to an exponential lower bound.
By
\citet[Theorem~5]{jordan2023deterministic}, there is a deterministic first-order
algorithm that finds a $(\delta,\eps)$-Goldstein stationary point of
$\tilde f$ with only logarithmic dependence on $L_1$. Applying this
algorithm on top of a smoothing procedure therefore finds a
$(p_1(\delta,\eps),p_2(\delta,\eps))$-Goldstein stationary point of $f$
using $\tilde{O}\left(\frac{T}{\delta\eps^3}\log\left(1+L_1\right)\right)$ oracle calls to $f$.
Letting $\delta,\eps>0$ be sufficiently small, dimension-independent
constants such that $p_1(\delta,\eps)<1$ and
$p_2(\delta,\eps)<1/32$, apply Theorem~\ref{thm: main} to the resulting
deterministic algorithm to see that
\[
    T\log(1+L_1)\gtrsim \exp(\Omega(d))~.
\]
Since meaningful smoothing requires $L_1\leq\operatorname{poly}(d)$,
we conclude that:

\begin{corollary} \label{cor: smoothing}
Any meaningful deterministic smoothing has complexity
$T\gtrsim\exp(\Omega(d))$.
\end{corollary}

\begin{remark}\label{rem: integral}
It is interesting to note the analogy with the curse of dimensionality
in numerical integration. For example,
\citet[Proposition~3.2]{hinrichs2014curse} show that deterministically
integrating $1$-Lipschitz functions $h:[0,1]^d\to[-1,1]$ requires
at least $\exp(\Omega(d))$ function evaluations, even for a fixed
constant accuracy. In contrast, averaging the function values at $n$ independent uniformly random points yields a root-mean-square error at most $1/\sqrt{n}$,
independently of $d$, allowing randomized integration with dimension-free oracle complexity. This analogy is particularly natural for
convolution-based smoothing, whose evaluation amounts to computing
a high-dimensional integral. However, the corollary above is not
restricted to evaluating a prescribed convolution, as it establishes exponential hardness for any meaningful deterministic smoothing.
\end{remark}

\subsection{Analogy to volume estimation}

Finally, following the spirit of Remark~\ref{rem: integral}, we discuss an even looser analogy that we find thought-provoking nonetheless.
Another computational task for which randomization provably offers an exponential advantage is that of estimating the volume of a convex body given membership oracle queries.
In particular, estimating $\mathrm{vol}(\Kcal)$ for a convex body $\Kcal\subset\reals^d$
to a fixed accuracy requires
$\exp(\Omega(d))$ deterministic oracle queries \citep{furedi1986computing}; however, a randomized algorithm can achieve this goal using only
$\operatorname{poly}(d)$ queries.
\citep{dyer1991computing,dyer1991random}.

This analogy has an intuitive, albeit informal, geometric interpretation.
As discussed in Section~\ref{sec: single step}, algorithms seeking Goldstein stationarity require ``exploring'' the Goldstein subdifferential of each iterate, in order to find a descent direction. Moreover, note that any convex body $\Kcal\subset\reals^d$ can be realized as a Goldstein subdifferential of its support function $\Kcal=\partial_\delta \sigma_{\Kcal}(\bzero)$.
Therefore, at a geometric level, both problems illustrate the difficulty of extracting certain global information (volume or shortest vector) from deterministic queries in high dimensions,
essentially due to the fact that covering numbers in $\reals^d$ scale exponentially with $d$.
We stress that the analogy should not be interpreted as a formal reduction, and we are not aware of one.

\subsection*{Statement on AI use}

The results in this paper were derived entirely by the author.
After preparing the manuscript, and before making it public, we asked ChatGPT 5.6 Sol to review it. In addition to various minor writing suggestions, the AI found a simplification of the hard function construction based on support functions, which is the version we ultimately chose to present here, after thoroughly revising and rewriting the details ourselves.
Our original construction was similar in terms of its properties and intuition, yet it boiled down to an explicit modification procedure around iterates (similarly to both \citealp{jordan2023deterministic,tian2024no}) which was considerably more tedious to analyze.
In our view, the modified construction is more elegant, and it also helps streamline the proofs of the further extensions in Section~\ref{sec: discuss}.
The author takes full responsibility for the paper’s content.


\bibliographystyle{plainnat}
\bibliography{bib}

@article{li2011concise,
  title={Concise Formulas for the Area and Volume of a Hyperspherical Cap},
  author={Li, S},
  journal={Asian Journal of Mathematics \& Statistics},
  volume={4},
  number={1},
  pages={66--70},
  year={2011},
  publisher={Science Alert}
}

@inproceedings{jordan2023deterministic,
  title={Deterministic nonsmooth nonconvex optimization},
  author={Jordan, Michael and Kornowski, Guy and Lin, Tianyi and Shamir, Ohad and Zampetakis, Manolis},
  booktitle={The Thirty Sixth Annual Conference on Learning Theory},
  pages={4570--4597},
  year={2023},
  organization={PMLR}
}

@article{tian2024no,
  title={No dimension-free deterministic algorithm computes approximate stationarities of Lipschitzians},
  author={Tian, Lai and So, Anthony Man-Cho},
  journal={Mathematical Programming},
  volume={208},
  number={1},
  pages={51--74},
  year={2024},
  publisher={Springer}
}

@inproceedings{zhang2020complexity,
  title={Complexity of finding stationary points of nonconvex nonsmooth functions},
  author={Zhang, Jingzhao and Lin, Hongzhou and Jegelka, Stefanie and Sra, Suvrit and Jadbabaie, Ali},
  booktitle={International Conference on Machine Learning},
  pages={11173--11182},
  year={2020},
  organization={PMLR}
}

@book{Clarke-1990-Optimization,
  title={Optimization and nonsmooth analysis},
  author={Clarke, Frank H},
  year={1990},
  publisher={SIAM}
}

@article{goldstein1977optimization,
  title={Optimization of Lipschitz continuous functions},
  author={Goldstein, Allen A},
  journal={Mathematical Programming},
  volume={13},
  pages={14--22},
  year={1977},
  publisher={Springer}
}

@book{Nemirovski-1983-Problem,
	author = {Nemirovski, Arkadi Semenovich and Yudin, David Borisovich},
	title  = {Problem complexity and method efficiency in optimization.},
	publisher = {Wiley},
	year   = {1983}
}

@article{murty1987some,
  title={Some NP-complete problems in quadratic and nonlinear programming.},
  author={Murty, Katta G and Kabadi, Santosh N},
  journal={Mathematical programming},
  volume={39},
  number={2},
  pages={117--129},
  year={1987}
}

@article{kornowski2022oracle,
  title={Oracle complexity in nonsmooth nonconvex optimization},
  author={Kornowski, Guy and Shamir, Ohad},
  journal={Journal of Machine Learning Research},
  volume={23},
  number={314},
  pages={1--44},
  year={2022}
}

@article{ghadimi2013stochastic,
  title={Stochastic first-and zeroth-order methods for nonconvex stochastic programming},
  author={Ghadimi, Saeed and Lan, Guanghui},
  journal={SIAM Journal on Optimization},
  volume={23},
  number={4},
  pages={2341--2368},
  year={2013},
  publisher={SIAM}
}

@article{carmon2020lower,
  title={Lower bounds for finding stationary points {I}},
  author={Carmon, Yair and Duchi, John C and Hinder, Oliver and Sidford, Aaron},
  journal={Mathematical Programming},
  volume={184},
  number={1-2},
  pages={71--120},
  year={2020},
  publisher={Springer}
}

@article{davis2020stochastic,
  title={Stochastic subgradient method converges on tame functions},
  author={Davis, Damek and Drusvyatskiy, Dmitriy and Kakade, Sham and Lee, Jason D},
  journal={Foundations of computational mathematics},
  volume={20},
  number={1},
  pages={119--154},
  year={2020},
  publisher={Springer}
}

@article{davis2022gradient,
  title={A gradient sampling method with complexity guarantees for Lipschitz functions in high and low dimensions},
  author={Davis, Damek and Drusvyatskiy, Dmitriy and Lee, Yin Tat and Padmanabhan, Swati and Ye, Guanghao},
  journal={Advances in neural information processing systems},
  volume={35},
  pages={6692--6703},
  year={2022}
}

@inproceedings{Tian-2022-Finite,
  title={On the finite-time complexity and practical computation of approximate stationarity concepts of lipschitz functions},
  author={Tian, Lai and Zhou, Kaiwen and So, Anthony Man-Cho},
  booktitle={International Conference on Machine Learning},
  pages={21360--21379},
  year={2022},
  organization={PMLR}
}

@inproceedings{cutkosky2023optimal,
  title={Optimal stochastic non-smooth non-convex optimization through online-to-non-convex conversion},
  author={Cutkosky, Ashok and Mehta, Harsh and Orabona, Francesco},
  booktitle={International Conference on Machine Learning},
  pages={6643--6670},
  year={2023},
  organization={PMLR}
}

@article{lin2022gradient,
  title={Gradient-free methods for deterministic and stochastic nonsmooth nonconvex optimization},
  author={Lin, Tianyi and Zheng, Zeyu and Jordan, Michael},
  journal={Advances in Neural Information Processing Systems},
  volume={35},
  pages={26160--26175},
  year={2022}
}

@inproceedings{chen2023faster,
  title={Faster gradient-free algorithms for nonsmooth nonconvex stochastic optimization},
  author={Chen, Lesi and Xu, Jing and Luo, Luo},
  booktitle={International Conference on Machine Learning},
  pages={5219--5233},
  year={2023},
  organization={PMLR}
}

@article{kornowski2024algorithm,
  title={An algorithm with optimal dimension-dependence for zero-order nonsmooth nonconvex stochastic optimization},
  author={Kornowski, Guy and Shamir, Ohad},
  journal={Journal of Machine Learning Research},
  volume={25},
  number={122},
  pages={1--14},
  year={2024},
}

@inproceedings{liu2024zeroth,
  title={Zeroth-Order Methods for Constrained Nonconvex Nonsmooth Stochastic Optimization},
  author={Liu, Zhuanghua and Chen, Cheng and Luo, Luo and Low, Bryan Kian Hsiang},
  booktitle={International Conference on Machine Learning},
  pages={30842--30872},
  year={2024},
  organization={PMLR}
}

@inproceedings{zhang2023private,
  title={Private Zeroth-Order Nonsmooth Nonconvex Optimization},
  author={Zhang, Qinzi and Tran, Hoang and Cutkosky, Ashok},
  booktitle={The Twelfth International Conference on Learning Representations},
  year={2024}
}

@inproceedings{kornowski2025improved,
  title={Improved Sample Complexity for Private Nonsmooth Nonconvex Optimization},
  author={Kornowski, Guy and Liu, Daogao and Talwar, Kunal},
  booktitle={International Conference on Machine Learning},
  pages={31541--31559},
  year={2025},
  organization={PMLR}
}

@article{grimmer2025goldstein,
  title={Goldstein stationarity in Lipschitz constrained optimization: B. Grimmer, Z. Jia},
  author={Grimmer, Benjamin and Jia, Zhichao},
  journal={Optimization Letters},
  volume={19},
  number={2},
  pages={425--435},
  year={2025},
  publisher={Springer}
}

@article{liu2024gradient,
  title={Gradient-free methods for nonconvex nonsmooth stochastic compositional optimization},
  author={Liu, Zhuanghua and Luo, Luo and Low, Bryan Kian Hsiang},
  journal={Advances in Neural Information Processing Systems},
  volume={37},
  pages={45438--45461},
  year={2024}
}

@article{shi2026nonsmooth,
  title={Nonsmooth Nonconvex-Concave Minimax Optimization: Convergence Criteria and Algorithms},
  author={Shi, Jinyang and Luo, Luo},
  journal={arXiv preprint arXiv:2604.21371},
  year={2026}
}

@inproceedings{lin2024decentralized,
  title={Decentralized gradient-free methods for stochastic non-smooth non-convex optimization},
  author={Lin, Zhenwei and Xia, Jingfan and Deng, Qi and Luo, Luo},
  booktitle={Proceedings of the AAAI Conference on Artificial Intelligence},
  volume={38},
  number={16},
  pages={17477--17486},
  year={2024}
}

@article{guan2026computing,
  title={On computing Goldstein approximate second-order stationary points of structured nonsmooth nonconvex programs},
  author={Guan, Jiewen and So, Anthony Man-Cho},
  journal={arXiv preprint arXiv:2607.24122},
  year={2026}
}

@article{daniilidis2020pathological,
  title={Pathological subgradient dynamics},
  author={Daniilidis, Aris and Drusvyatskiy, Dmitriy},
  journal={SIAM Journal on Optimization},
  volume={30},
  number={2},
  pages={1327--1338},
  year={2020},
  publisher={SIAM}
}

@Article{Burke-2002-Approximating,
    Title           = {Approximating subdifferentials by random sampling of gradients},
    Author          = {Burke, James V and Lewis, Adrian S and Overton, Michael L},
    Journal         = {Mathematics of Operations Research},
    Volume          = {27},
    Number          = {3},
    Pages           = {567-584},
    Year            = {2002},
    Publisher       = {INFORMS}
}

@Article{Burke-2005-Robust,
    Title           = {A robust gradient sampling algorithm for nonsmooth, nonconvex optimization},
    Author          = {Burke, James V and Lewis,  Adrian S and Overton, Michael L},
    Journal         = {SIAM Journal on Optimization},
    Volume          = {15},
    Number          = {3},
    Pages           = {751-779},
    Year            = {2005},
    Publisher       = {SIAM}
}

@Article{Kiwiel-2007-Convergence,
    Title           = {Convergence of the gradient sampling algorithm for nonsmooth nonconvex optimization},
    Author          = {Kiwiel, Krzysztof C },
    Journal         = {SIAM Journal on Optimization},
    Volume          = {18},
    Number          = {2},
    Pages           = {379-388},
    Year            = {2007},
    Publisher       = {SIAM}
}

@article{Burke-2020-Gradient,
  title={Gradient sampling methods for nonsmooth optimization},
  author={Burke, James V and Curtis, Frank E and Lewis, Adrian S and Overton, Michael L and Sim{\~o}es, Lucas EA},
  journal={Numerical nonsmooth optimization: State of the art algorithms},
  pages={201--225},
  year={2020},
  publisher={Springer}
}

@InProceedings{zhang2024randomscaling,
  title = 	 {Random Scaling and Momentum for Non-smooth Non-convex Optimization},
  author =       {Zhang, Qinzi and Cutkosky, Ashok},
  booktitle = 	 {Proceedings of the 41st International Conference on Machine Learning},
  pages = 	 {58780--58799},
  year = 	 {2024},
  volume = 	 {235},
  series = 	 {Proceedings of Machine Learning Research},
  month = 	 {21--27 Jul},
  publisher =    {PMLR},
}

@article{ahn2024adam,
  title={Adam with model exponential moving average is effective for nonconvex optimization},
  author={Ahn, Kwangjun and Cutkosky, Ashok},
  journal={Advances in Neural Information Processing Systems},
  volume={37},
  pages={94909--94933},
  year={2024}
}

@InProceedings{jiang2026adaptive,
  title = 	 {Adaptive Matrix Online Learning through Smoothing with Guarantees for Nonsmooth Nonconvex Optimization},
  author =       {Jiang, Ruichen and Mhammedi, Zakaria and Mohri, Mehryar and Mokhtari, Aryan},
  booktitle = 	 {Proceedings of Thirty Ninth Conference on Learning Theory},
  pages = 	 {3782--3824},
  year = 	 {2026},
  volume = 	 {336},
  series = 	 {Proceedings of Machine Learning Research},
  month = 	 {29 Jun--03 Jul},
  publisher =    {PMLR},
}

@InProceedings{ahn2025general,
  title = 	 {General framework for online-to-nonconvex conversion: Schedule-free {SGD} is also effective for nonconvex optimization},
  author =       {Ahn, Kwangjun and Magakyan, Gagik and Cutkosky, Ashok},
  booktitle = 	 {Proceedings of the 42nd International Conference on Machine Learning},
  pages = 	 {772--795},
  year = 	 {2025},
  volume = 	 {267},
  series = 	 {Proceedings of Machine Learning Research},
  month = 	 {13--19 Jul},
  publisher =    {PMLR},
}

@article{li2026muon,
  title={Muon with Finite Newton-Schulz: The Smoothing Benefit in Nonsmooth Nonconvex Optimization},
  author={Li, Mingyi and Tsuchiya, Taira},
  journal={arXiv preprint arXiv:2608.26288},
  year={2026}
}

@inproceedings{flaxman2005online,
  title={Online convex optimization in the bandit setting: gradient descent without a gradient},
  author={Flaxman, Abraham D and Kalai, Adam Tauman and McMahan, H Brendan},
  booktitle={Proceedings of the sixteenth annual ACM-SIAM symposium on Discrete algorithms},
  pages={385--394},
  year={2005}
}

@article{yousefian2012stochastic,
  title={On stochastic gradient and subgradient methods with adaptive steplength sequences},
  author={Yousefian, Farzad and Nedi{\'c}, Angelia and Shanbhag, Uday V},
  journal={Automatica},
  volume={48},
  number={1},
  pages={56--67},
  year={2012},
  publisher={Elsevier}
}

@article{duchi2012randomized,
  title={Randomized smoothing for stochastic optimization},
  author={Duchi, John C and Bartlett, Peter L and Wainwright, Martin J},
  journal={SIAM Journal on Optimization},
  volume={22},
  number={2},
  pages={674--701},
  year={2012},
  publisher={SIAM}
}

@article{shamir2017optimal,
  title={An optimal algorithm for bandit and zero-order convex optimization with two-point feedback},
  author={Shamir, Ohad},
  journal={Journal of Machine Learning Research},
  volume={18},
  number={52},
  pages={1--11},
  year={2017}
}

@article{hinrichs2014curse,
  title={The curse of dimensionality for numerical integration of smooth functions II},
  author={Hinrichs, Aicke and Novak, Erich and Ullrich, Mario and Wo{\'z}niakowski, Henryk},
  journal={Journal of Complexity},
  volume={30},
  number={2},
  pages={117--143},
  year={2014},
  publisher={Elsevier}
}

@inproceedings{furedi1986computing,
  title={Computing the volume is difficult},
  author={Furedi, Z and Barany, I},
  booktitle={Proceedings of the eighteenth annual ACM symposium on Theory of computing},
  pages={442--447},
  year={1986}
}

@article{dyer1991computing,
  title={Computing the volume of convex bodies: a case where randomness provably helps},
  author={Dyer, Martin and Frieze, Alan},
  journal={Probabilistic combinatorics and its applications},
  volume={44},
  number={123-170},
  pages={0754--68052},
  year={1991},
  publisher={American Mathematical Society Providence, RI}
}

@article{dyer1991random,
  title={A random polynomial-time algorithm for approximating the volume of convex bodies},
  author={Dyer, Martin and Frieze, Alan and Kannan, Ravi},
  journal={Journal of the ACM (JACM)},
  volume={38},
  number={1},
  pages={1--17},
  year={1991},
  publisher={ACM New York, NY, USA}
}

\appendix

\section{Spherical cap measure bound}

\begin{lemma} \label{lem: cap measure}
For $k\in\NN$, let $\mu_k$ be the normalized surface measure on $\SS^{k}$, and let $\bu\in\SS^{k}$. Then for any $r\in(0,1):~
\mu_k(\{\bz\in\SS^k:\|\bz-\bu\|\leq r\})
\leq (2r)^k$.
\end{lemma}

\begin{proof}[Proof of Lemma~\ref{lem: cap measure}]

We can assume that $r\leq \half$, since otherwise the statement is trivial (as the right hand side is larger than $1$).
A standard formula for the spherical cap measure (e.g., \citealp{li2011concise}) is
\[
\mu_k(\{\bz\in\SS^k:\|\bz-\bu\|\leq r\})
=\frac{\int_{0}^{2\arcsin(r/2)}\sin^{k-1}(s)ds}{\int_{0}^{\pi}\sin^{k-1}(s)ds}~.
\]
Bounding the numerator and the denominator separately, we have
\begin{align*}
\int_{0}^{2\arcsin(r/2)}\sin^{k-1}(s)ds
\leq \int_{0}^{2\arcsin(r/2)} s^{k-1} ds
=\frac{(2\arcsin(r/2))^k}{k}
\leq \frac{(2r)^k}{k}
~,
\end{align*}
and
\[
\int_{0}^{\pi}\sin^{k-1}(s)ds
=2\int_{0}^{\pi/2}\sin^{k-1}(s)ds
\geq 2\int_{0}^{\pi/2}\left(\frac{2s}{\pi}\right)^{k-1}ds
=2\left(\frac{2}{\pi}\right)^{k-1}\frac{(\pi/2)^k}{k}
=\frac{\pi}{k}>\frac{2}{k}~,
\]
so overall
\[
\mu_k(\{\bz\in\SS^k:\|\bz-\bu\|\leq r\})
\leq \frac{(2r)^k/k}{2/k}\leq (2r)^k~.
\]

\end{proof}

\end{document}